\documentclass[11pt]{amsart}

\usepackage{amssymb}
\usepackage[abbrev]{amsrefs}

\allowdisplaybreaks
\numberwithin{equation}{section}
\newcommand{\N}{\mathbb N}
\newcommand{\R}{\mathbb R}
\theoremstyle{plain}
\newtheorem{lemma}{Lemma}[section]
\newtheorem{theorem}[lemma]{Theorem}
\newtheorem{corollary}[lemma]{Corollary}
\newtheorem{example}[lemma]{Example}
\theoremstyle{remark}
\newtheorem{remark}[lemma]{Remark}

\title[Fixed point theorem for a Meir-Keeler type mapping]%
{Fixed point theorem for a Meir-Keeler type mapping in a metric
space with a transitive relation}

\author{Koji~Aoyama}
\address[K.~Aoyama]
{Aoyama Mathematical Laboratory,
Konakadai, Inage-ku, Chiba, Chiba 263-0043, Japan}
\email{aoyama@bm.skr.jp}

\author{Masashi~Toyoda}
\address[M.~Toyoda]
{Department of Information Science, Toho University, 
Miyama, Funabashi, Chiba 274-8510, Japan}
\email{mss-toyoda@is.sci.toho-u.ac.jp}


\begin{document}

\begin{abstract}
 The aim of this paper is to provide characterizations of a Meir-Keeler 
 type mapping and a fixed point theorem for the mapping 
 in a metric space endowed with a transitive relation. 
\end{abstract}

\maketitle

\section{Introduction}

Let $X$ be a metric space with metric $d$,
$R$ a subset of $X\times X$, and $T\colon X \to X$ a mapping. 
We say that $T$ is a \emph{Meir-Keeler type} mapping on $R$
if for any $\epsilon > 0$ there exists $\delta >0$ such that
\[
(x,y)\in R \text{ and } \epsilon \leq d(x,y) < \epsilon + \delta 
\text{ imply } d(Tx,Ty) < \epsilon.  
\]
This mapping is based on a mapping introduced in Meir and
Keeler~\cite{MR0250291}. 
Indeed, a Meir-Keeler type mapping $T$ on $X\times X$ is 
a \emph{weakly uniformly strict contraction} in the sense of
\cite{MR0250291}, which is often called a 
\emph{Meir-Keeler contraction}. 

In Section~\ref{s:characterization}, 
we provide some characterizations of a Meir-Keeler type mapping
(Theorem~\ref{t:MKC-char-R}). 
The result includes characterizations of a Meir-Keeler contraction by 
Wong~\cite{MR644645}, Lim~\cite{MR1845580}, and Gavruta et
al.~\cite{gavruta2014two}. 

In Section~\ref{s:fpt}, we establish a fixed point theorem for a
Meir-Keeler type mapping (Theorem~\ref{t:fpt}) in a metric space
endowed with a transitive relation. 
The result is related to the study of Ben-El-Mechaiekh~\cite{MR3346760}
and fixed point theorems in a metric space with a
partial order proved in 
Ran and Reurings~\cite{MR2053350}, 
Nieto and Rodr\'{\i}guez-L\'{o}pez~\cite{MR2212687}, and
Reich and Zaslavski~\cite{reich2017monotone}.

\section{Preliminaries}

Throughout the present paper, 
$\N$ denotes the set of positive integers, 
$\R$ the set of real numbers, 
and $\R_+$ the set of nonnegative real numbers.

A function $l \colon \R_+ \to \R_+$ 
is said to be of \emph{type (L)}
if for any $s>0$ there exists $\delta >0$ such that
$l(t) \leq s$ for all $t \in [s, s+\delta]$. 
It is clear that if a function $l\colon \R_+ \to \R_+$ is of type (L),
then $l(t) \leq t$ for all $t > 0$. 

\begin{remark}
 A mapping of type (L) above is based on an \textit{L}-function 
 introduced in \cite{MR1845580}. 
 We say that a function $l \colon \R_+ \to \R_+$ 
 is an \emph{\textit{L}-function} \cite{MR1845580}
 if $l(0)=0$, $l(s)>0$ for all $s>0$, and $l$ is of type (L). 
\end{remark}

We say that a function $w\colon \R_+ \to \R$ is 
\emph{right lower semicontinuous}
at $t_0\in \R_+$ if
for any $\epsilon > 0$ there exists $\delta > 0$ such that
$w(t_0) - \epsilon < w(s)$ for all $s \in [t_0, t_0 + \delta)$; 
a function $\psi\colon \R_+ \to \R$ is 
\emph{right upper semicontinuous} at $t_0 \in \R_+$
if $-\psi$ is right lower semicontinuous at $t_0$. 
It is clear that if $w\colon \R_+ \to \R$ is a nondecreasing
function, then $w$ is right lower semicontinuous at any $t \in \R_+$. 
It is known that a function $w\colon \R_+ \to \R$ is 
right lower semicontinuous at $t_0\in \R_+$ if and only if
$w(t_0) \leq \liminf_n w(s_n)$ 
whenever $\{s_n\}$ is a sequence in $[t_0,\infty)$ 
such that $s_n \to t_0$. 

\section{Characterizations of a Meir-Keeler type mapping}
\label{s:characterization}

The aim of this section is to prove the following theorem, 
which provides characterizations of a Meir-Keeler type mapping
defined on a metric space endowed with a transitive relation. 

\begin{theorem}\label{t:MKC-char-R}
 Let $X$ be a metric space with metric $d$,
 $T\colon X \to X$ a mapping, 
 and $R$ a nonempty subset of $X \times X$. 
 Then the following are equivalent: 
 \begin{enumerate}
  \item $T$ is a Meir-Keeler type mapping on $R$, that is, 
	for any $\epsilon > 0$ there exists $\delta >0$ such that
	$(x,y)\in R$ and $\epsilon \leq d(x,y) < \epsilon + \delta$ imply
	$d(Tx,Ty) < \epsilon$;
  \item for any $\epsilon > 0$ there exists $\delta >0$ such that
	$(x,y)\in R$ and $d(x,y) < \epsilon + \delta$
	imply $d(Tx,Ty) < \epsilon$; 
  \item \label{Gamma:R}
       there exists a nondecreasing function 
       $\gamma \colon \R_+ \to [0,\infty]$ such that 
       $\gamma(s) > s$ for all $s>0$ and 
       $\gamma\bigl( d(Tx,Ty) \bigr) \leq d(x,y)$ for all $(x,y) \in R$;
  \item \label{Wong:R}
	there exists a function $w\colon \R_+ \to \R_+$ such that 
	$w(s) > s$ for all $s>0$, 
	$w$ is right lower semicontinuous on $(0,\infty)$, 
	and 
	$w\bigl( d(Tx,Ty) \bigr) \leq d(x,y)$ for all $(x,y) \in R$;
  \item \label{Lim:R}
	there exists a function $l\colon (0,\infty) \to \R_+$ of 
	type~(L) such that $d(Tx,Ty) < l \bigl( d(x,y) \bigr)$ for all 
	$(x,y) \in R$ with $x \ne y$;
  \item \label{Phi-Psi:R}
	there exist a nondecreasing function 
	$\phi \colon \R_+ \to [0,\infty]$ and a function 
	$\psi \colon \R_+ \to \R_+$ such that 
	$\psi$ is right upper semicontinuous on $(0,\infty)$, 
	$\phi(t) > \psi(t)$ for all $t >0$, and 
	$\phi \bigl( d(Tx,Ty) \bigr) \leq \psi \bigl( d(x,y) \bigr)$
	for all	$(x,y) \in R$. 
 \end{enumerate}
 Moreover, in \eqref{Lim:R}, one can choose $l$ to be a right continuous
 and nondecreasing function such that $l(s) > 0$ for all $s>0$. 
\end{theorem}

Obviously, Theorem~\ref{t:MKC-char-R} is valid in case of $R=X\times
X$. Therefore Theorem~\ref{t:MKC-char-R} provides characterizations 
of a Meir-Keeler contraction \cite{MR0250291} on a metric space. 

\begin{remark}
 The condition~(\ref{Gamma:R}) is related to the modulus of uniform
 continuity of $T$; see Lim~\cite{MR1845580}. 
 The conditions~(\ref{Wong:R}) and~(\ref{Lim:R}) are based on
 \cite{MR1845580}*{Theorem~1}; see also Wong~\cite{MR644645} 
 for~(\ref{Wong:R}). 
 The condition~(\ref{Phi-Psi:R}) comes from a \emph{weak type
 contraction} introduced in \cite{gavruta2014two}. 
\end{remark}

Theorem~\ref{t:MKC-char-R} above is a direct consequence of
Theorem~\ref{t:fg} below. 
We first prove it by using lemmas in Section~\ref{s:lemmas}. 

\begin{theorem}\label{t:fg}
 Let $K$ be a nonempty set and let $f\colon K \to \R_+$ and $g\colon K
 \to \R_+$ be functions. 
 Suppose that $g^{-1}(0) \subset f^{-1}(0)$. 
 Then the following are equivalent: 
 \begin{enumerate}
  \item \label{MK:fg}
	For any $\epsilon >0$ there exists $\delta > 0$ such
	that $x \in K$ and $\epsilon \leq g(x) < \epsilon + \delta$
	imply $f(x) < \epsilon$;
  \item \label{MKs:fg}
	for any $\epsilon >0$ there exists $\delta > 0$ such that
	$x\in K$ and $g(x) < \epsilon + \delta$ imply $f(x) < \epsilon$; 
  \item \label{Gamma:fg}
       there exists a nondecreasing function 
       $\gamma \colon \R_+ \to [0,\infty]$ such that 
       $\gamma(s) > s$ for all $s>0$ and 
       $\gamma\bigl( f(x) \bigr) \leq g(x)$ for all $x \in K$;
  \item \label{Wong-finite:fg}
       there exists a function $w\colon \R_+ \to \R_+$ such that 
       $w(s) > s$ for all $s>0$, 
       $w$ is right lower semicontinuous on $(0,\infty)$, 
       and $w\bigl( f(x) \bigr) \leq g(x)$ for all $x \in K$;
  \item \label{Lim:fg}
	there exists a function $l\colon (0,\infty) \to \R_+$ of type
	(L) such
	that $f(x) < l \bigl( g(x) \bigr)$ for all $x \in K$ with 
	$g(x) \ne 0$. 
  \item \label{Phi-Psi:fg}
	there exist a nondecreasing function $\phi \colon \R_+ \to
	[0,\infty]$ and a function $\psi \colon \R_+ \to \R_+$
	such that $\phi(t) > \psi(t)$ for all $t >0$, 
	$\psi$ is right upper semicontinuous on $(0,\infty)$, 
	and 
	$\phi \bigl( f(x) \bigr) \leq \psi \bigl( g(x) \bigr)$
	for all $x \in K$.
 \end{enumerate} 
 Moreover, 
 in \eqref{Lim:fg}, one can choose $l$ to be a right continuous and
 nondecreasing function such that $l(s) > 0$ for all $s>0$. 
\end{theorem}

\begin{proof}
 The implications \eqref{MKs:fg} $\Rightarrow$ \eqref{MK:fg}
 and \eqref{Gamma:fg} $\Rightarrow$ \eqref{Phi-Psi:fg} are clear. 
 Lemma~\ref{l:MK=Lim} shows that \eqref{MK:fg} and \eqref{Lim:fg} are
 equivalent,
 and that $l$ in \eqref{Lim:fg} can be chosen to be a right continuous and
 nondecreasing function such that $l(s) > 0$ for all $s>0$.
 Lemmas~\ref{l:mks2gamma}, \ref{l:gamma2w-finite}, 
 and~\ref{l:w-finite2MKs} show the implications
 \eqref{MKs:fg} $\Rightarrow$ \eqref{Gamma:fg}, 
 \eqref{Gamma:fg} $\Rightarrow$ \eqref{Wong-finite:fg}, 
 and \eqref{Wong-finite:fg} $\Rightarrow$ \eqref{MKs:fg}, 
 respectively. 
 Moreover, the implication \eqref{MK:fg} $\Rightarrow$ \eqref{MKs:fg}
 and \eqref{Phi-Psi:fg} $\Rightarrow$ \eqref{MK:fg}
 follow from Lemmas~\ref{l:MK2MKs} and~\ref{l:Phi-Psi2MK}, 
 respectively. 
 This completes the proof. 
\end{proof}

The following example shows that the implication \eqref{MK:fg}
$\Rightarrow$ \eqref{MKs:fg} in Theorem~\ref{t:fg} does not hold without
the assumption $g^{-1}(0) \subset f^{-1}(0)$. 

\begin{example}\label{e:MK/=MKs}
 Let $K = \{x\}$ be a singleton and 
 let $f\colon K \to \R_+$ and $g\colon K \to \R_+$ be functions
 defined by $f(x) = 1$ and $g(x)=0$.
 Then \eqref{MK:fg} in Theorem~\ref{t:fg} holds, 
 but \eqref{MKs:fg} in Theorem~\ref{t:fg} does not hold. 
\end{example}

\begin{proof}
 Let $\epsilon=1$. Then 
 $0 = g(x) < \epsilon + \delta$ and $f(x) \geq \epsilon$
 for all $\delta > 0$. Thus \eqref{MKs:fg} does not hold. 
 On the other hand, 
 let $\epsilon > 0$ and $\delta =1$. Then 
 $\{y \in K\colon \epsilon \leq g(y) < \epsilon + \delta \} = \emptyset$. 
 Therefore, \eqref{MK:fg} does hold. 
\end{proof}

\begin{remark}
 Let $K$, $f$, and $g$ be the same as in Example~\ref{e:MK/=MKs}
 and let $\phi \colon \R_+ \to [0,\infty]$ and 
 $\psi \colon \R_+ \to \R_+$ be functions defined by 
 $\phi (t) \equiv 1/2$ and 
 \[
 \psi (t) = \begin{cases}
	     1 & \text{ if } t=0; \\
	     1/4 & \text{otherwise}.
	    \end{cases}
 \]
 Then $\phi$ is nondecreasing, $\psi$ is right upper
 semicontinuous on $(0,\infty)$, and  $\phi(t)> \psi(t)$ for all
 $t>0$. Since 
 \[
 \phi \bigl( f(x) \bigr) = \phi(1) = 1/2 \leq 1 = \psi(0) 
 = \psi \bigl( g(x) \bigr), 
 \]
 it follows that 
 $\phi \bigl( f(y) \bigr) \leq \bigl( g(y) \bigr)$ for all $y \in K$. 
 Therefore  Example~\ref{e:MK/=MKs} also shows that the implication
 \eqref{Phi-Psi:fg} $\Rightarrow$ \eqref{MKs:fg} in Theorem~\ref{t:fg}
 does not hold without the assumption $g^{-1}(0) \subset f^{-1}(0)$. 
\end{remark}

Using Theorem~\ref{t:fg}, 
we can easily obtain Theorem~\ref{t:MKC-char-R}. 

\begin{proof}[Proof of Theorem~\ref{t:MKC-char-R}]
 Let $f\colon R \to \R_+$ and $g\colon R \to \R_+$ be functions defined
 by $f(x,y) = d(Tx, Ty)$ and $g(x,y)=d(x,y)$ for $(x,y) \in R$.
 Then it is clear that $g^{-1}(0) \subset f^{-1}(0)$. 
 Therefore Theorem~\ref{t:fg} implies the conclusion. 
\end{proof}

\section{Fixed point theorems} \label{s:fpt}

The aim of this section is to establish fixed point theorems for a
Meir-Keeler type mapping defined on a complete metric space endowed with 
a transitive relation or a partial order. 

\begin{theorem}\label{t:fpt}
 Let $X$ be a complete metric space with metric $d$,
 $T\colon X \to X$ a mapping, 
 and $R$ a nonempty subset of $X \times X$. Suppose that
 \begin{enumerate}
  \item \label{i:transitive}
	$(u,v) \in R$ and $(v,w) \in R$ imply $(u,w) \in R$; 
  \item \label{i:x}
	there exists $x \in X$ such that $(x,Tx) \in R$; 
  \item \label{i:R2R}
	$(Tu,Tv) \in R$ for all $(u,v)\in R$; 
  \item \label{i:MKC}
	for any $\epsilon >0$ there exists $\delta >0$ such that
	$(u,v) \in R$ and $\epsilon \leq d (u, v) < \epsilon + \delta$
	imply $d(Tu, Tv) < \epsilon$; 
  \item \label{i:pseudoc}
	if $\{x_n \}$ is a sequence in $X$ such that $x_n \to y$ and 
	$(x_n, x_{n+1}) \in R$ for all $n \in \N$, 
	then there exists a subsequence $\{x_{n_k}\}$ of $\{x_n\}$ 
	such that $T x_{n_k} \to Ty$ as $k \to \infty$. 
 \end{enumerate}
 Then $\{T^n x\}$ converges to a fixed point of $T$, that is, $T$ has a
 fixed point. Moreover, suppose that 
 \begin{enumerate}
  \setcounter{enumi}{5}
  \item \label{i:forally} $(x,y) \in R$ for all $y \in X$; 
  \item \label{i:Rclosed} $R$ is closed in $X\times X$. 
 \end{enumerate}
 Then $T$ has a unique fixed point. 
\end{theorem}

\begin{remark}
 The assumptions \eqref{i:forally} and \eqref{i:Rclosed} in
 Theorem~\ref{t:fpt} can be
 replaced by the following condition: 
 \begin{quote}
  If $y$ is a fixed point of $T$, and $\{x_n\}$ is a sequence in $X$
  such that $x_n \to z \in X$ and $(x_n,y)\in R$ for all $n \in \N$,
  then $(z,y) \in R$. 
 \end{quote}
\end{remark}

To prove Theorem~\ref{t:fpt}, we need lemmas below,
which are based on the results in \cite{MR0250291}*{\S2}. 

\begin{lemma}\label{l:nonincreasing}
 Let $X$ be a metric space with metric $d$,
 $T\colon X \to X$ a mapping, $x \in X$, and
 $\{x_n\}$ a sequence in $X$ defined by $x_n = T^n x$ for $n \in \N$.
 Suppose that for any $\epsilon >0$ there exists $\delta >0$ such that
 \begin{equation}\label{e:MKseq}
  n \in \N, \,
   \epsilon \leq d (x_n, x_{n+1}) < \epsilon + \delta
   \Rightarrow d(x_{n+1}, x_{n+2}) < \epsilon.
 \end{equation}
 Then $\left\{ d ( x_n, x_{n+1} ) \right\}$ is nonincreasing
 and $\lim_n d ( x_n, x_{n+1} ) = 0$. 
\end{lemma}

\begin{proof}
 Suppose that $d(x_m, x_{m+1}) = 0$. Then $x_m = x_{m+1}$.
 Thus we have $x_{m+1} = T^{m+1}x = T x_m =Tx_{m+1} =x_{m+2}$,
 and hence $d(x_{m+1}, x_{m+2}) = 0$. 
 On the other hand, suppose that $\epsilon = d(x_m, x_{m+1}) > 0$. 
 Then there exists $\delta > 0$ such that \eqref{e:MKseq} holds.
 Thus we have $d(x_{m+1},x_{m+2}) < \epsilon = d(x_m, x_{m+1})$. 
 Consequently, we know that 
 $\left\{ d ( x_n, x_{n+1} ) \right\}$ is nonincreasing, 
 and hence $\lim_n d ( x_n, x_{n+1} )$ exists. 
 Suppose that $\epsilon = \lim_n d(x_n, x_{n+1}) > 0$. 
 Then there exists $\delta > 0$ such that 
 \eqref{e:MKseq} holds. 
 Since $d(x_n, x_{n+1}) \searrow \epsilon$, there exists $k \in \N$ such
 that $\epsilon \leq d(x_k,x_{k+1}) < \epsilon + \delta$. 
 Thus we have $\epsilon \leq d(x_{k+1}, x_{k+2}) < \epsilon$, which is a
 contradiction. 
 Therefore, $\lim_n d(x_n,x_{n+1}) = \epsilon = 0$.
\end{proof}

\begin{lemma}\label{l:pre-cauchy}
 Let $X$ be a metric space with metric $d$, $\{x_n\}$ a sequence in $X$, 
 $l,m$ positive integers, and $\epsilon,\eta$ positive real numbers.
 Suppose that $l<m$, $\eta \leq \epsilon$, 
 $d(x_l, x_m) \geq 2 \epsilon$, and $d(x_i, x_{i+1}) < \eta/3$ for all
 $i\in \N$ with $l \leq i \leq m$. 
 Then there exists $j \in \N$ such that $l< j <m$ and
 $\epsilon + 2\eta/3 \leq d(x_l,x_j) < \epsilon + \eta$. 
\end{lemma}

\begin{proof}
 Set $A= \{ i \in \N\colon l < i <m,\, 
 \epsilon + 2\eta/3 \leq d (x_l, x_i)\}$. 
 We first show that $m -1 \in A$. 
 Suppose that $m-1 \leq l$. Then $m = l+1$, and we have
 \[
 2 \epsilon \leq d(x_l, x_m) = d(x_l, x_{l+1}) < \eta/3 \leq \epsilon/3,
 \]
 which is a contradiction. Thus $l < m-1$. Moreover, we have
 \[
 d(x_l, x_{m-1}) \geq d(x_l,x_m) - d(x_m, x_{m-1})
 \geq 2\epsilon - \eta/3 \geq \epsilon + 2\eta/3.
 \]
 Therefore, $m-1 \in A$, and hence $A$ is nonempty. 

 Set $j = \min A$. 
 Suppose that $l \geq j-1$. Then $j=l+1$. Thus we have 
 $\epsilon + 2\eta/3 \leq d(x_l, x_j) = d(x_l,x_{l+1}) < \eta/3$, 
 which is a contradiction. 
 Therefore, 
 $l < j-1 < j < m$. 
 Since $j-1 \notin A$, we have 
 $d(x_l, x_{j-1}) < \epsilon +2\eta/3$, and hence
 \[
 d(x_l,x_j) \leq  d(x_l,x_{j-1}) + d(x_{j-1},x_j)
 < \epsilon + 2\eta/3 + \eta/3 = \epsilon + \eta. 
 \]
 As a result, we conclude that 
 $l< j <m$ and
 $\epsilon + 2\eta/3 \leq d(x_l,x_j) < \epsilon + \eta$. 
\end{proof}

\begin{lemma} \label{l:cauchy}
 Let $X$, $T$, $x$, and $\{x_n\}$ be the same as in
 Lemma~\ref{l:nonincreasing}. 
 Suppose that for any $\epsilon >0$ there exists $\delta >0$ such that
 \begin{equation}\label{e:MKseq2}
  i, j \in \N, \,
   \epsilon \leq d (x_i , x_j ) < \epsilon + \delta
   \Rightarrow d(x_{i+1}, x_{j+1}) < \epsilon.
 \end{equation}
 Then $\{x_n\}$ is a Cauchy sequence. 
\end{lemma}

\begin{proof}
 Suppose that $\{x_n\}$ is not a Cauchy sequence. 
 Then there exists $\epsilon> 0$ such that 
 for each $i \in \N$ there exist $m_i,n_i \in \N$ such that 
 \begin{equation}\label{e:not-cauchy}
  i \leq m_i < n_i \text{ and } d(x_{m_i}, x_{n_i}) \geq 2 \epsilon. 
 \end{equation}
 By assumption, we know that there exists $\delta>0$ such that 
 \eqref{e:MKseq2} holds. 
 Set $\eta = \min \{ \delta, \epsilon\}$. 
 Since $d(x_n,x_{n+1}) \searrow 0$ by Lemma~\ref{l:nonincreasing},
 it follows from~\eqref{e:not-cauchy} that there exist $m, n \in \N$
 with $m<n$ such that $d(x_m, x_n) \geq 2 \epsilon$ and
 \begin{equation}\label{e:eta3}
  d(x_i, x_{i+1}) < \eta/3
 \end{equation}
 for all $i \in \N$ with $i\geq m$.   
 Thus Lemma~\ref{l:pre-cauchy} shows that 
 there exists $j \in \N$ such that $m<j<n$ and 
 \[
  \epsilon + 2\eta/3 \leq d (x_m, x_j) < \epsilon + \eta.   
 \]
 As a result, we see that 
 $\epsilon \leq d (x_m, x_j) < \epsilon + \delta$. 
 Taking into account \eqref{e:eta3} and \eqref{e:MKseq2}, we have
 \begin{align*}
  \epsilon + 2\eta/3 \leq d(x_m, x_j)
  &\leq  d(x_m, x_{m+1}) + d(x_{m+1}, x_{j+1}) + d(x_{j+1},x_j)\\
  &< \eta/3 + \epsilon + \eta/3 =  \epsilon + 2\eta/3, 
 \end{align*}
 which is a contradiction. Therefore, $\{x_n\}$ is a Cauchy sequence. 
\end{proof}

Now we prove Theorem~\ref{t:fpt}. 

\begin{proof}[Proof of Theorem~\ref{t:fpt}]
 Let $\{x_n\}$ be a sequence in $X$ defined by $x_n = T^n x$ for $n \in
 \N$. 
 Then, by the assumptions~(\ref{i:transitive}), (\ref{i:x}), and
 (\ref{i:R2R}),
 we see that $(x_m, x_n) \in R$ for all $m,n \in \N$ 
 with $m < n$.
 Thus it follows from the assumption~(\ref{i:MKC}) that
 for any $\epsilon >0$ there exists $\delta >0$ such that
 \eqref{e:MKseq2} holds. 
 Since $X$ is complete, Lemma~\ref{l:cauchy} shows that
 $\{x_n\}$ converges to some point $z \in X$. 
 We show that $z$ is a fixed point of $T$. 
 By virtue of the assumption \eqref{i:pseudoc}, there exists a
 subsequence $\{x_{n_k}\}$ of $\{x_n\}$ such that $Tx_{n_k} \to Tz$ as
 $k \to \infty$. Taking into account $x_{n_k + 1}\to z$, we conclude that
 \[
 d(Tz,z) \leq d(Tz, x_{n_k + 1}) + d(x_{n_k + 1}, z) 
 = d(Tz, Tx_{n_k}) + d(x_{n_k + 1}, z) \to 0
 \]
 as $k\to \infty$. Therefore, $Tz=z$, and hence $z$ is a fixed point of
 $T$. 

 We next show that $z$ is the unique fixed point of $T$
 under the assumptions \eqref{i:forally} and \eqref{i:Rclosed}. 
 Let $y$ be a fixed point of $T$. 
 Since $(x,y) \in R$ by~\eqref{i:forally}, it follows from~\eqref{i:R2R}
 that $(Tx,y) = (Tx, Ty) \in R$. 
 Therefore, $(T^n x,y) \in R$ for all $n \in \N$. 
 Since $T^n x \to z$ and $R$ is closed by~\eqref{i:Rclosed}, 
 we conclude that $(z,y) \in R$. 
 Using Theorem~\ref{t:MKC-char-R} 
 and the function $\gamma$ in Theorem~\ref{t:MKC-char-R} \eqref{i:R2R}, 
 we have
 \[
 \gamma \bigl( d(z,y) \bigr) = \gamma \bigl( d(Tz,Ty) \bigr) \leq d(z,y),
 \]
 and hence $z=y$. 
\end{proof}

Using Theorem~\ref{t:fpt}, we obtain the following: 

\begin{corollary}[%
 Nieto \& Rodr\'{i}guez-L\'{o}pez \cite{MR2212687}*{Theorem 2.2}]
 Let $X$ be a complete metric space with metric $d$,
 $T\colon X \to X$ a mapping, 
 and $\preceq$ a partial order in $X$. Suppose that
 \begin{enumerate}
  \item[(NR1)] there exists $x \in X$ such that $x \preceq Tx$; 
  \item[(NR2)] $Tu \preceq Tv$ for all $u,v\in X$ with $u \preceq v$;	
  \item[(NR3)] there exists $\theta \in [0,1)$ such that
	       $d(Tu, Tv) \leq \theta d(u,v)$ for all $u,v\in X$ with 
	       $u \preceq v$; 
  \item[(NR4)] if $\{x_n \}$ is a sequence in $X$ such that $x_n \to y$
	       and $x_n \preceq x_{n+1}$ for all $n \in \N$, 
	       then then $x_n \preceq y$  for all $n \in \N$. 
 \end{enumerate}
 Then $T$ has a fixed point.
\end{corollary}

\begin{proof}
 Set $R = \{ (u,v) \in X \times X\colon u \preceq v\}$.
 Since $(x,Tx) \in R$ by~(NR1), we know that $R$ is a nonempty subset of
 $X\times X$ and the assumption~\eqref{i:x} in
 Theorem~\ref{t:fpt} holds. 
 The assumption \eqref{i:transitive} in Theorem~\ref{t:fpt} 
 is valid clearly. 
 The assumptions~\eqref{i:R2R} and~\eqref{i:MKC} in
 Theorem~\ref{t:fpt} follow from (NR2) and (NR3),
 respectively. 
 We must check the assumption~\eqref{i:pseudoc} in
 Theorem~\ref{t:fpt}. 
 Let $\{x_n \}$ be a sequence in $X$ such that $x_n \to y$ and 
 $(x_n, x_{n+1}) \in R$ for all $n \in \N$. 
 Taking into account (NR3) and (NR4), we see that 
 \[
  d(Tx_n, Ty) \leq \theta d(x_n,y) \to 0
 \]
 as $n \to \infty$. 
 Therefore Theorem~\ref{t:fpt} implies the conclusion. 
\end{proof}

Using Theorem~\ref{t:fpt}, we also deduce the following fixed point
theorem, which is similar to \cite{reich2017monotone}*{Theorem 1.2}. 

\begin{theorem}
 Let $Y$ be a complete metric space with metric $d$,
 $\preceq$ a partial order in Y, $X$ a nonempty closed subset of $Y$, 
 and $T\colon X \to X$ a mapping. Suppose that
 \begin{itemize}
  \item[(RZ0)] $\{(u,v) \in Y\times Y \colon u \preceq v\}$ is closed in
	       $Y\times Y$; 
  \item[(RZ1)] the graph of $T$ is closed in $Y\times Y$; 
  \item[(RZ2)] $Tu \preceq Tv$ for all $u,v \in X$ with $u \preceq v$;
  \item[(RZ3)] there exists a right upper semicontinuous function
	       $\psi\colon \R_+ \to \R_+$ such that $t > \psi(t)$ for
	       all $t > 0$ and $d(Tu,Tv) \leq \psi \bigl( d(u,v) \bigr)$
	       for all $u,v \in X$ with $u \preceq v$; 
  \item[(RZ4)] there exists $x \in X$ such that $x \preceq y$ for all
	       $y \in X$.
 \end{itemize}
 Then $\{T^n x\}$ converges to a unique fixed point of $T$. 
\end{theorem}

\begin{proof}
 By assumption, it is clear that $X$ is complete. 
 Set $R = \{(u,v) \in X\times X \colon u \preceq v\}$. 
 By virtue of (RZ4), $(x, x) \in R$, and hence $R$ is nonempty. 
 Moreover, since $\preceq$ is a partial order,  
 the assumption \eqref{i:transitive} in Theorem~\ref{t:fpt} holds. 
 The assumptions \eqref{i:x} and \eqref{i:forally} in
 Theorem~\ref{t:fpt} follow from (RZ4); 
 the assumption \eqref{i:R2R} in Theorem~\ref{t:fpt} follows from
 (RZ2). 
 Since $X$ is closed, the assumption \eqref{i:Rclosed} in
 Theorem~\ref{t:fpt} is deduced from (RZ0). 
 Using Theorem~\ref{t:MKC-char-R}, 
 we know that (RZ3) implies
 the assumption \eqref{i:MKC} in Theorem~\ref{t:fpt}. 
 Therefore it is enough to verify the assumption \eqref{i:pseudoc} in
 Theorem~\ref{t:fpt}. 
 Let $\{x_n \}$ be a sequence in $X$ such that 
 $x_n \to y$ and $(x_n, x_{n+1}) \in R$ for all $n \in \N$. 
 Since $X$ is closed, it follows that $y \in X$. 
 Let $m \in \N$ be fixed. 
 Then it is easy to check that $(x_m, x_n) \in R$ for all $n \in \N$
 with $m \leq n$. 
 Since $\{(x_m, x_n) \}_{n\geq m}$ converges to $(x_m,y)$ in 
 $X\times X$ and $R$ is closed in $X\times X$,  
 we see that $(x_m, y) \in R$. 
 Hence $(x_m, y) \in R$ for all $m \in \N$. 
 Set $A = \{n \in \N \colon x_n = y \}$. 
 Suppose that $A$ is an infinite set.
 Then there exists a subsequence $\{x_{n_k}\}$ of $\{x_n\}$ 
 such that $x_{n_k} = y$ for all $k \in \N$, 
 and hence $Tx_{n_k} \to Ty$ as $k\to \infty$. 
 On the other hand, suppose that $A$ is not a infinite set. 
 Then there exists a subsequence $\{x_{n_k}\}$ of $\{x_n\}$ 
 such that $x_{n_k} \ne y$ for all $k \in \N$. 
 Since $(x_{n_k}, y) \in R$ and $d(x_{n_k}, y) > 0$ for all $k \in \N$,
 it follows from~(RZ3) that
 \[
 d(Tx_{n_k}, Ty) \le \psi \bigl( d(x_{n_k},y) \bigr) 
 < d(x_{n_k},y) \to 0
 \]
 as $k \to \infty$.  
 Therefore the assumption \eqref{i:pseudoc} in Theorem~\ref{t:fpt}
 holds.   
 Consequently, Theorem~\ref{t:fpt} implies the conclusion. 
\end{proof}

\section{Lemmas}\label{s:lemmas}

In this section, we prove lemmas which are used in 
the proof of Theorem~\ref{t:fg}. 

In what follows, let $K$ be a nonempty set and 
let $f\colon K \to \R_+$ and $g\colon K \to \R_+$ be functions. 

\begin{lemma}\label{l:MK=Lim}
 The conditions \eqref{MK:fg} and \eqref{Lim:fg} 
 in Theorem~\ref{t:fg} are equivalent. 
 Moreover, in \eqref{Lim:fg}, one can choose $l$ to be a right
 continuous and
 nondecreasing function such that $l(s) > 0$ for all $s>0$. 
\end{lemma}

\begin{proof}
 We first prove \eqref{Lim:fg} $\Rightarrow$ \eqref{MK:fg}. 
 Let $\epsilon> 0$. Since $l$ is of type (L), there exists $\delta >0$
 such that $l(t) \leq \epsilon$ for all 
 $t \in [\epsilon, \epsilon + \delta]$. 
 Let $x \in K$ with $\epsilon \leq g(x) < \epsilon + \delta$. 
 Then $g(x) \ne 0$. 
 Thus it follows from~\eqref{Lim:fg} that 
 $f(x) < l \bigl( g(x) \bigr) \leq \epsilon$. 

 We next prove \eqref{MK:fg} $\Rightarrow$ \eqref{Lim:fg}
 and the ``Moreover'' part. 
 We follow the proof of \cite{MR2196804}*{Proposition 1}. 
 By assumption, for any $\epsilon>0$ there exists $\alpha(\epsilon) >0$
 such that 
 \begin{equation}\label{183605}
  x\in K,\, \epsilon \leq g(x) < \epsilon + 2 \alpha(\epsilon)
 \Rightarrow f(x) < \epsilon.
 \end{equation}
 Since 
 $\{ \epsilon > 0\colon t \leq \epsilon + \alpha(\epsilon)\}\ne
 \emptyset$ for all $t>0$, 
 we can define a function $\beta \colon (0,\infty) \to [0,\infty)$ by
 \[
 \beta(t) = \inf \{ \epsilon > 0\colon t \leq \epsilon +
 \alpha(\epsilon)\}
 \]
 for $t > 0$.
 Then it is clear that $\beta$ is nondecreasing, 
 $\beta(t) \leq t$ for all $t >0$,  and moreover, 
 $\min \{ \epsilon > 0\colon t \leq \epsilon +
 \alpha(\epsilon)\}$ exists for all $t>0$ with $\beta(t) = t$. 
 Let $\phi_1 \colon (0,\infty) \to [0,\infty)$ be a function defined by
 \[
 \phi_1(t) = \begin{cases}
	      \beta(t) & \text{if } \min \{ \epsilon >0 \colon 
	      t \leq \epsilon + \alpha(\epsilon)\} \text{ exists}; \\
	      \dfrac{\beta(t) + t}2 & \text{otherwise}
	     \end{cases}
 \]
 for $t > 0$.
 Then we verify the following: 
 \begin{itemize}
  \item[(i)] $\phi_1(t) > 0$ for all $t>0$;
  \item[(ii)] $\phi_1$ is of type~(L);
  \item[(iii)] $f(x) < \phi_1 \bigl( g(x) \bigr)$
	       for all $x \in K$ with $g(x) \ne 0$. 
 \end{itemize}
 By the definition of $\phi_1$, (i) is clear.
 We show (ii). Let $s > 0$ be fixed. 
 Suppose that $\phi_1 (t) \leq s$ for all $t \in (s,s+\alpha(s)]$. 
 Then setting $\delta = \alpha(s)$, we conclude that 
 \begin{equation}\label{152553}
  t \in [s, s+ \delta] \Rightarrow \phi_1(t) \leq s.
 \end{equation}
 On the other hand, 
 suppose that there exists $\sigma \in (s, s+ \alpha(s)]$ such that
 $\phi_1(\sigma) > s$. Then $s \in \{\epsilon>0 \colon \sigma \leq
 \epsilon + \alpha(\epsilon)\}$, and hence $\beta(\sigma) \leq s$. 
 If $\beta(\sigma) = s$, then we have
 $\beta(\sigma) = \min \{ 
 \epsilon> 0\colon \sigma \leq \epsilon + \alpha (\epsilon)\}$, 
 and thus
 \[
  \phi_1 (\sigma) = \beta(\sigma) = s < \phi_1 (\sigma),
 \]
 which is a contradiction. Consequently, we know that 
 \[
  \beta (\sigma) < s < \phi_1(\sigma) = \dfrac{\beta(\sigma) + \sigma}2.
 \]
 Taking into account the definition of $\beta(\sigma)$, 
 we can choose $u \in [\beta(\sigma), s)$ with $\sigma \leq u +
 \alpha(u)$.
 Then set $\delta = s-u$ and let $t \in [s, s+\delta]$. 
 Since 
 \[
  t \leq s + \delta = 2s - u < 2\cdot \dfrac{\beta(\sigma) + \sigma}2 
 - \beta(\sigma) = \sigma \leq u+ \alpha(u),
 \]
 it follows that $\beta(t)\leq u$. 
 Therefore we have
 \[
  \phi_1 (t) \leq \dfrac{\beta(t) + t}2 \leq \dfrac{u+s+\delta}2 = s.
 \]
 Thus \eqref{152553} holds, and hence
 $\phi_1$ is of type (L). 
 We next show (iii). 
 Let $x \in K$ with $g(x) \ne 0$. 
 Taking into account the definition of $\phi_1$, 
 we know that for any $t>0$ 
 there exists $\epsilon \in (0, \phi_1(t)]$ such that 
 $\epsilon \leq t \leq \epsilon + \alpha (\epsilon)$,
 and thus there exists 
 $\epsilon \in \left(0, \phi_1 \bigl( g(x) \bigr) \right]$ such
 that $\epsilon \leq g(x) \leq \epsilon + \alpha (\epsilon)$. 
 Hence we deduce from \eqref{183605} that 
 $f(x) < \epsilon \leq \phi_1 \bigl( g(x) \bigr)$. 
 Consequently, (iii) holds. 
 Now let us define functions $\phi_2\colon (0,\infty) \to \R_+$ and 
 $l \colon (0,\infty) \to \R_+$ by
 \[
 \phi_2(t)= \sup\{\phi_1(s) \colon s \leq t\}
 \text{ and }
 l(t)= \inf\{\phi_2(s) \colon s > t\}
 \]
 for $t\in (0,\infty)$. 
 Then it is not hard to check that 
 $\phi_2$ and $l$ are well-defined and nondecreasing, 
 and moreover, 
 \[
 0 < \phi_1(t) \leq \phi_2(t)  \leq l(t) \leq t
 \]
 for all $t>0$.
 Thus it follows from~(ii) and~(iii) that
 $l$ is of type~(L)
 and $f(x) < l \bigl( g(x) \bigr)$ for all $x \in K$ with 
 $g(x) \ne 0$. 
 We can also verify that 
 $l$ is right continuous.
 This completes the proof. 
\end{proof}

\begin{lemma}\label{l:mks2gamma}
 The condition \eqref{MKs:fg} in Theorem~\ref{t:fg} implies
 the condition \eqref{Gamma:fg} in Theorem~\ref{t:fg}. 
\end{lemma}

\begin{proof}
 Define a function $\gamma \colon \R_+ \to [0,\infty]$ by
 \[
 \gamma(t) = \inf \{g(x)\colon x\in K,\, f(x)\geq t\}
 \]
 for $t \in \R_+$, where $\inf \emptyset = \infty$. 
 Then the function $\gamma$ is well-defined and nondecreasing, 
 and moreover, 
 $\gamma \bigl( f(x) \bigr) \leq g(x)$ for all $x \in K$. 
 Hence it is enough to show that $\gamma(t) > t$ for all $t > 0$. 
 Suppose that $\gamma(t) \leq t$ for some $t >0$. 
 Then, by assumption, there exists $\delta>0$ such that
 $x \in K$ and $g(x) < t + \delta$ imply $f(x) < t$. 
 Since $\gamma(t) < t + \delta$, 
 there exists $y \in K$ such that $f(y)\geq t$ and $g(y)< t + \delta$.
 Therefore we have $t \leq f(y) < t$, which is a contradiction. 
\end{proof}

\begin{lemma}\label{l:gamma2w-finite}
 The condition \eqref{Gamma:fg} in Theorem~\ref{t:fg}
 implies the condition \eqref{Wong-finite:fg} in Theorem~\ref{t:fg}. 
\end{lemma}

\begin{proof}
 We follow the idea of the proof of \cite{MR1845580}*{Theorem 1}. 
 If $\{ t \in \R_+ \colon \gamma(t) = \infty\}$ is empty, 
 then we easily obtain the conclusion. 
 Thus we may assume that 
 $\{ t \in \R_+ \colon \gamma(t) = \infty\}$ is nonempty. 
 Set $t_0 = \inf \{ t \in \R_+ \colon \gamma(t) = \infty\}$.
 In the case of $\gamma(t_0) < \infty$, 
 let $w_1 \colon \R_+ \to \R_{+}$ be a function defined by
 \[
 w_1 (t) = \begin{cases}
	 \gamma(t) & \text{if }t\in [0,t_0]; \\
	 \gamma(t_0) + t - t_0 & \text{otherwise}.
	\end{cases}
 \]
 Then it is clear that 
 $w_1(s) > s$ for all $s>0$.
 Since $w_1$ is nondecreasing,
 we know that $w_1$ is right lower
 semicontinuous on $(0,\infty)$. 
 We can also check that 
 $w_1 \bigl( f(x)\bigr) \leq g(x)$ for all $x \in K$.
 On the other hand, in the case of $\gamma(t_0) = \infty$, 
 let $w_2 \colon \R_+ \to \R_{+}$ be a function defined by
 \[
 w_2(t) = \begin{cases}
	 \gamma(t) & \text{if }t\in [0,t_0), \\
	 2t & \text{otherwise}. 
	\end{cases}
\]
 Then it is clear that 
 $w_2(s) > s$ for all $s>0$.
 Since 
 $w_2$ is nondecreasing on $(0,t_0)$
 and continuous on $[t_0,\infty)$,
 we know that $w_2$ is right lower semicontinuous on $(0,\infty)$. 
 We can also check that 
 $w_2 \bigl( f(x)\bigr) \leq g(x)$ for all $x \in K$.
\end{proof}

\begin{lemma}\label{l:w-finite2MKs}
 The condition \eqref{Wong-finite:fg} in Theorem~\ref{t:fg}
 implies the condition  \eqref{MKs:fg} in Theorem~\ref{t:fg}. 
\end{lemma}

\begin{proof}
 Suppose that \eqref{MKs:fg} does not hold. 
 Then there exist $\epsilon>0$ and a sequence $\{x_n\}$ in $K$ such that 
 $g(x_n)<  \epsilon + 1/n$ and $f(x_n) \geq \epsilon$ for all $n \in \N$. 
 Since $f(x_n) > 0$, it follows from the properties of $w$ that
 \[
  \epsilon \leq f(x_n) < w\bigl( f(x_n)\bigr) \leq g(x_n) <
 \epsilon + 1/n
 \]
 for all $n \in \N$. 
 Hence $f(x_n) \to \epsilon$ and $w\bigl( f(x_n) \bigr) \to \epsilon$. 
 Since $w$ is right lower semicontinuous at $\epsilon$ and 
 $\epsilon < w(\epsilon)$, 
 we have 
 $\epsilon < w(\epsilon) \leq \liminf_n w\bigl( f(x_n) \bigr) =
 \epsilon$, which is a contradiction. 
\end{proof}

\begin{lemma}\label{l:MK2MKs}
 Suppose $g^{-1}(0) \subset f^{-1} (0)$. 
 Then the condition \eqref{MK:fg} in Theorem~\ref{t:fg}
 implies the condition \eqref{MKs:fg} in Theorem~\ref{t:fg}. 
\end{lemma}

\begin{proof}
 Let $\epsilon > 0$ be given. 
 Then, by~\eqref{MK:fg},  
 there exists $\delta > 0$ such that 
 $x \in K$ and $\epsilon \leq g(x) < \epsilon + \delta$ imply
 $f(x) < \epsilon$. 
 Let $x \in K$ such that $g(x) < \epsilon$. 
 It is enough to show that $f(x) < \epsilon$. 
 Suppose that $g(x)=0$. Then, by assumption,  $f(x) = 0 < \epsilon$. 
 On the other hand, suppose that $0< g(x) < \epsilon$. 
 Set $\epsilon' = g(x)$. Then, by \eqref{MK:fg}, 
 there exists $\delta' > 0$ such that $y \in K$ 
 and $\epsilon' \leq g(y) < \epsilon' + \delta'$
 imply $f(y) < \epsilon'$. 
 Since $\epsilon' = g(x) < \epsilon' + \delta'$, 
 we have $f(x) < \epsilon' = g(x) < \epsilon$. 
\end{proof}

\begin{lemma}\label{l:Phi-Psi2MK}
 The condition \eqref{Phi-Psi:fg} in Theorem~\ref{t:fg} 
 implies the condition \eqref{MK:fg} in Theorem~\ref{t:fg}. 
\end{lemma}

\begin{proof}
 Suppose that \eqref{MK:fg} does not hold. 
 Then there exist $\epsilon >0$ and a sequence $\{x_n\}$ in $K$
 such that $\epsilon \leq g(x_n) < \epsilon + 1/n$ and $f(x_n) \geq
 \epsilon$ for all $n \in \N$.
 Thus $g(x_n) \to \epsilon$ and, by assumption, 
 \[
 \psi(\epsilon) < \phi(\epsilon) \leq \phi \bigl( f(x_n) \bigr) \leq
 \psi \bigl( g(x_n) \bigr)  
 \]
 for all $n \in \N$.
 Since $\psi$ is right upper semicontinuous at $\epsilon$, 
 we conclude that 
 $\psi(\epsilon) < \phi(\epsilon) \leq \limsup_n \psi \bigl( g(x_n) \bigr)
 \leq \psi(\epsilon)$, which is a contradiction.
\end{proof}

\section*{Acknowledgment}

The first author would like to acknowledge the financial support from
Professor Kaoru Shimizu of Chiba University.

\end{document}